\pdfoutput=1
\documentclass[11pt,reqno]{amsart}
\usepackage[a4paper,margin=.95in,footskip=0.3in]{geometry}
\usepackage{mathrsfs}
\usepackage{amssymb}
\usepackage{mathtools}
\usepackage{tikz-cd}
\usepackage{enumitem}
\usepackage[all,cmtip]{xy}
\usepackage[square, comma, numbers, sort&compress]{natbib}
\usepackage[T1]{fontenc}
\usepackage{booktabs}
\usepackage{pgf,tikz}
\usepackage{tikz-cd}
\usepackage{subcaption}
\usepackage{setspace}
\usepackage{csquotes}
\usepackage{epigraph}
\usepackage{extarrows}
\usetikzlibrary{arrows}
\definecolor{qqqqff}{rgb}{0,0,1}
\usepackage{lscape}
\usepackage{bm}
\usepackage{todonotes}

\PassOptionsToPackage{pdfusetitle,pagebackref,colorlinks}{hyperref}
\usepackage{bookmark}
\hypersetup{
  linkcolor={red!70!black},
  citecolor={blue!70!black},
  urlcolor={blue!80!black}
}
\usepackage[capitalize]{cleveref}
\makeatletter
  \def\th@plain{
  \thm@headfont{\bfseries} 
  \thm@notefont{\itshape} 
  \itshape
}
\makeatother

\makeatletter
  \def\th@definition{
  \thm@headfont{\bfseries} 
  \thm@notefont{\bfseries} 
}
\makeatother

\makeatletter
  \def\th@remark{
  \thm@headfont{\bfseries} 
  \thm@notefont{\bfseries} 
	}
\makeatother

\newtheorem{theorem}{Theorem}[section]
\newtheorem{lemma}[theorem]{Lemma}
\newtheorem{proposition}[theorem]{Proposition}
\newtheorem{corollary}[theorem]{Corollary}

\theoremstyle{question}
\newtheorem{question}[theorem]{Question}
\theoremstyle{definition}
\newtheorem{definition}[theorem]{Definition}

\theoremstyle{remark}
\newtheorem{remark}[theorem]{Remark}

\newtheoremstyle{cited}{.5\baselineskip\@plus.2\baselineskip\@minus.2\baselineskip}{.5\baselineskip\@plus.2\baselineskip\@minus.2\baselineskip}{\itshape}{}{\bfseries}{\bfseries .}{5pt plus 1pt minus 1pt}{\thmname{#1}\thmnumber{~#2}\thmnote{ \normalfont#3}}
\theoremstyle{cited}
\newtheorem{citedthm}[theorem]{Theorem}

\newtheorem{citedlem}[theorem]{Lemma}
\newtheorem{citedprop}[theorem]{Proposition}

\newtheoremstyle{citeddef}{.5\baselineskip\@plus.2\baselineskip\@minus.2\baselineskip}{.5\baselineskip\@plus.2\baselineskip\@minus.2\baselineskip}{}{}{\bfseries}{\bfseries .}{5pt plus 1pt minus 1pt}{\thmname{#1}\thmnumber{~#2}\thmnote{ \normalfont#3}}
\theoremstyle{citeddef}

\title[symplectic birational self-maps of HKs of K$3^{[n]}$-type]{On symplectic birational self-maps of projective hyperk\"{a}hler manifolds of K3$^{[n]}$-type}
\date{}
\author{Yajnaseni Dutta}
\address{Universiteit Leiden, Mathematisch Instituut, Niels Bohrweg 1, 2333 CA, Leiden, Netherlands}
\email{y.dutta@math.leidenuniv.nl}

\author{Dominique Mattei}

\address{Universit\"at Bonn, Endenicher Allee 60, 53115 Bonn, Germany}
\email{dmattei@math.uni-bonn.de}

\author{Yulieth Prieto--Montañez}

\address{The Abdus Salam International Centre for Theoretical Physics, Str. Costiera, 11, 34151 Trieste TS, Italy }
\email{yprieto@ictp.it}

\newcommand\hq{\mathbb Q}
\newcommand\hz{\mathbb Z}
\newcommand\Z{\mathbb Z}

\newcommand\ho{\mathrm{O}}

\newcommand\F{\mathcal{F}}

\newcommand{\function}{\longrightarrow}

\newcommand{\rk}{\operatorname{rank}}

\newcommand{\NS}{\operatorname{NS}}

\newcommand{\sign}{\text{sign}}

\newcommand{\dive}{\text{div}}

\newcommand{\Exc}{\text{Exc }}

\newcommand{\ign}[1]{}

\newcommand{\Br}{\operatorname{Br}}
\newcommand{\LambdaKn}{\Lambda_{\operatorname{K3}^{[n]}}}
\newcommand{\ord}{\operatorname{ord}}

\newcommand{\MukaiLat}{\widetilde{\Lambda}}

\DeclareMathOperator{\Aut}{Aut}

\DeclareMathOperator{\Pic}{Pic}

\DeclareMathOperator{\Mov}{Mov}

\DeclareMathOperator{\Tr}{T} 

\DeclareMathOperator{\Amp}{Amp}
\DeclareMathOperator{\Nef}{Nef}
\DeclareMathOperator{\Mon}{Mon^2}
\DeclareMathOperator{\MonB}{Mon^2_{Bir}}
\DeclareMathOperator{\MonH}{Mon^2_{Hdg}}
\DeclareMathOperator{\Bir}{Bir}

\DeclareMathOperator{\D}{D}

\DeclareMathOperator{\Stab}{Stab}
\DeclareMathOperator{\Stabd}{Stab^\dagger}

\DeclareMathOperator{\Coh}{Coh} 

\newcommand\matR{\mathbb{R}}
\newcommand\matQ{\mathbb{Q}}

\newcommand\matZ{\mathbb{Z}}

\newcommand\calA{\mathcal{A}}

\newcommand\calC{\mathcal{C}}

\newcommand\calE{\mathcal{E}}
\newcommand\calF{\mathcal{F}}

\newcommand\calH{\mathcal{H}}

\newcommand\calO{\mathcal{O}}

\newcommand\calT{\mathcal{T}}

\newcommand\calW{\mathcal{W}}

\newcommand{\sH}{\mathcal{H}}

\newcommand{\sW}{\mathcal{W}}

\newcommand{\bbC}{\mathbb{C}}

\newcommand{\bbQ}{\mathbb{Q}}
\newcommand{\bbR}{\mathbb{R}}
\newcommand{\bbZ}{\mathbb{Z}}

\DeclareMathOperator{\dv}{div}
\DeclareMathOperator{\Id}{Id}

\begin{document}

	\begin{abstract}
		We prove that projective hyperk\"{a}hler manifolds of K3$^{[n]}$-type admitting a non-trivial symplectic birational self-map of finite order
        are isomorphic to moduli spaces of stable (twisted) coherent sheaves on K3 surfaces. 
		Motivated by this result, we analyze the reflections on the movable cone of moduli spaces of sheaves 
        and determine when they come from a birational involution.
	\end{abstract}

 \maketitle
	
	\setcounter{tocdepth}{1}
	\tableofcontents
	
	\section{Introduction}
	
	A hyperk\"ahler manifold (\textsl{HK manifold} for short) is a simply connected compact K\"ahler manifold $X$ admitting an everywhere (unique up-to scalar) non-degenerate holomorphic $2$-form $\sigma$ that generates $H^0(X,\Omega_X^2)$. We say that a HK manifold is of K3$^{[n]}$-type if it is a smooth deformation of the Hilbert scheme of $n\geq 2$ points on a K3 surface. The goal of this paper is to study (symplectic) birational self-maps of finite order on such HK manifolds. 
	Our first result identifies HK manifolds of K3$^{[n]}$-type admitting such maps.
	
	\begin{theorem} \label{MainThmSympBirImpliesModuli}
		Let $X$ be a projective HK manifold of K3$^{[n]}$-type admitting a non-trivial finite order symplectic birational self-map. 
		Then $X$ is isomorphic to a moduli space of (twisted) sheaves on a K3 surface.
	\end{theorem}

	In particular, $X$ has Picard rank at least $2$. In other words, a very general projective HK manifold of K3$^{[n]}$-type cannot admit symplectic finite order birational self-maps. This follows also from \cite[Prop.\ 4.3]{DebarreHKMan}, which in addition confirms that our result is optimal in the sense that certain very general HK manifolds of K$3^{[n]}$-type do indeed admit finite order automorphisms, however, these are always anti-symplectic. 
	
	In fact, the proof of Theorem \ref{MainThmSympBirImpliesModuli} (in \S \ref{SectionCohActionSymplecticBirMaps}) reveals that if the finite order birational map is biregular (i.e.\ an automorphism) or if the order is larger than 2, then the rank of the Picard group of $X$ must be at least 8.
	
	When the Picard rank is $2$, it turns out that all such maps are automatically involutions. 
	\begin{theorem}\label{PropPicRk2AllInvolutions}
		Let $X$ be a HK manifold of K$3^{[n]}$-type. Assume that $X$ has Picard rank $2$. Then any non-trivial finite order birational self-map has order 2. Furthermore, if such a map is an automorphism, it must be anti-symplectic. 
	\end{theorem}
	
	It is known that whenever the group $\Bir(X)$ of birational self-maps is finite,  $\Bir(X)\simeq (\matZ/2\matZ)^{\oplus m}$ with $m=1$ or $2$ \cite[Thm.\ 4.9]{DebarreHKMan}. Moreover, for any K3 surface $S$ of Picard rank $1$, the group $\Bir(S^{[n]})$ is finite \cite[Cor.\ 5.2]{OguisoAutomGrpCYManPicRk2}.
	
	In the next part of the paper, using Bridgeland stability on K3 surfaces, we study the existence of finite order symplectic birational self-maps on moduli spaces of stable pure sheaves $X=  M(S,v)$ (and their generalizations) on K3 surfaces 
	for $v$ a primitive Mukai vector.

	Mongardi in \cite[Thm.\ 26]{Mon16} shows that any symplectic (biregular) automorphism of finite order acts as $\Id$ on the discriminant group $A_X$ (introduced in \S \ref{secpreliminaries} below). We deduce in Corollary \ref{CoroSympReflectionActsNonTrivDiscGrp} that if a birational self-map $g$ induces a reflection in $\ho(H^2(X,\matZ), q_X)$, then it follows that $g^*|_{A_X}=-\Id$. Hence $g$ cannot be regular everywhere. Birational self-maps that are not biregular are hard to come across (for instance, on K3 surfaces there are none). On the Hilbert scheme of points on K3 surfaces of Picard rank one, such maps have been classified in \cite{BeriCattaneoBiratTransfHilbSchemeK3Surf}. 
	Typical examples of such symplectic birational maps are described by Markman in \cite[\S 11.1, 11.2]{Mar13} (see \S \ref{SectionMarkmanExamples} for one of the examples). In our next main result, we provide a list of Mukai vectors on K3 surfaces that satisfy specific numerical conditions. This list corresponds to cases where symplectic birational (non-biregular) self-maps exist on the associated moduli space of (complexes of) sheaves.

	\begin{theorem}\label{MainThm2}
		Let $S$ be a K3 surface and $v=(r, cD, s)$, $D\in\NS(S)$, be a primitive Mukai vector, $r\neq 0$. Consider a moduli space $M\coloneqq  M_\sigma(S,v)$ of stable objects with respect to a $v$-generic stability condition $\sigma$. For $e = (r, cD, (cD)^2/r - s)$, the reflection $R_e\in O(H^2(M,\matZ))$ is induced by a birational involution if and only if $v$ satisfies $r\mid 2c$, $\gcd(r,s)=1$ or $2$, $r\neq 1$ or 2 and $v $ is not one of the following series of Mukai vectors
		\begin{enumerate}
			\item $(r, krD, \frac{k^2D^2r}{2} - m)$ for some $k\in\matZ$ and $m=1$ or $2$.
			\item $D^2\equiv 0 \pmod{4}$, $v=(2a,maD,\dfrac{D^2}{4}m^2a-1)$ 
			for some integers $a\geq 2$, $m$ odd. 
			\item $D^2\equiv 2 \pmod{4}$, $v=(2a,maD,\dfrac{D^2m^2a-2}{4})$
			for some integers $a\geq 2$ odd, $m$ odd.
		\end{enumerate}
	\end{theorem}

	The vector $e$ is orthogonal to the so-called \textit{vertical wall} in $\NS(M)_\matR$, see Definition \ref{def:verticalwall}.
	In fact, we note in Corollary \ref{verticalwallonetwo} that if $v$ falls into the last three cases, $M\simeq M(S',v')$ (up to a twist by a Brauer class) for a derived equivalent K3 surface $S'$ and a Mukai vector $v'= (r',c'D', s')$ such that $r' =1$ or 2, and $R_e$ induces a reflection $R_{e'}$ with $e' = (r', c'D', \frac{c'^2D'^2}{r'} - s')$. So in a sense, ranks 1 and 2 are the only problematic cases.

  A complete classification of Mukai vectors for which the associated moduli spaces admit symplectic involutions that are birational but not biregular will entail studying other relevant reflections in $\NS(S)_\matR$. To this end, we show in Corollary \ref{CoroAllReflectionAreVertical} that whenever such a reflection $R$ exists, after changing $S$ and $v$ to $S'$ and $v'$ under a Fourier--Mukai transform, we can assume that $M=M(S',v')$ and $R$ is as in Theorem \ref{MainThm2} with respect to $v'$. 
	
	\subsection{Preliminaries} \label{secpreliminaries}

        The purpose of this section is to provide a review of the key properties and constructions of HK manifolds, with a particular focus on those of K3$^{[n]}$-type. This overview will serve as a fundamental framework for the subsequent parts of the paper.
        
        A crucial tool in the study of HK manifolds is the second cohomology group $H^2(X,\mathbb{Z})$ endowed with the Beauville--Bogomolov--Fujiki form $q_X$. In the case where $X$ is a HK manifold of K3$^{[n]}$-type, the pair $(H^2(X,\mathbb{Z}),q_X)$ is an indefinite lattice of signature $(3, 20)$.
        We denote by $\ho(H^2(X,\bbZ),q_X)$ the group of isometries of $H^2(X,\bbZ)$ and by $\ho(A_X,\bar{q}_X)$ the group of isometries of the \textsl{discriminant group} $A_X \coloneqq  H^2(X,\bbZ)^{\vee}/H^2(X,\bbZ)$.

        Birational self-maps on HK manifolds naturally induce automorphisms on $H^2(X,\bbZ)$. In \cite[Lem.\ 2.6]{HuyCompactHK}, Huybrechts shows that the quadratic form $q_X$ behaves well under birational maps.  As a consequence, a well-defined group homomorphism exists:

		$$\Bir(X) \to \ho(H^2(X,\matZ), q_X).$$
        For HK manifolds of K$3^{[n]}$-type, this homomorphism is not only well-defined but also injective (\cite{BeauvilleSomeRksKahlerManifoldsC10}, \cite{HassetTschinkelHdgTheoryLagPlanesGenKummer}). We denote by $\MonB(X)$ its image.
        The group $\Mon(X)$ of monodromy operators, corresponding to parallel transport isometries of $H^2(X,\bbZ)$, is of great importance due to the birational Torelli theorem (\cite{VerbitskyMappingClassGrpGlobalTorelli}, \cite{VerbitskyErrataMappingClassGrpGlobalTorelli}, see \cite[Thm.\ $1.3$]{Mar13}). According to \cite[Lem.\ 4.2]{Mar10}, under the assumption that $X$ is of K$3^{[n]}$-type, the group $\Mon(X)$ can be identified as the preimage of $\{\Id,-\Id\}$ through the natural surjective map \cite[Prop.\ 1.14.2]{Nik79}:
		$$ \ho(H^2(X,\matZ), q_X) \to \ho(A_X, \bar{q}_X).$$

		The subgroup $\MonH(X)< \Mon(X)$ preserving the Hodge structure splits as the semidirect product
		$\MonH(X) = W_\text{Exc} \rtimes \MonB(X)$,
		where $W_\text{Exc}$ is the Weyl subgroup generated by reflections of the form
		$$R_E : \alpha \mapsto \alpha -\dfrac{2(\alpha,[E])}{([E],[E])}[E]$$
		where $E\in\Pic(X)$ is a \textsl{prime exceptional divisor} (i.e. reduced and irreducible effective divisors with $q_X(E)<0$). While $R_E$ is a priori only defined over $\matQ$, Markman proved that it is in fact an integral monodromy operator \cite{Mar13}.

	 A key input in this paper is Markman's detailed analysis of various isometries and domains in the N\'eron--Severi lattice $\NS(M)_{\bbR}$ in \cite{Mar11, Mar13}.
	 The \textsl{positive} cone $\calC_X\subset \NS(X)_\matR$ of classes $\alpha$ with $q_X(\alpha)\geq0$ contains the (non-necessarily closed) \textsl{movable} cone $\Mov(X)$, generated by divisors whose base locus has codimension at least $2$. The \textsl{nef} cone $\Nef(X)$ is contained in $\overline{\Mov(X)} \subset \calC_X$.

	 The closure $\overline{\Mov(X)}$ is a fundamental domain for the action of the Weil group $W_\Exc$ on $\calC_X$, while $\MonB(X)$ preserves $\Mov(X)$ (see \cite[Prop. $3.15$]{DebarreHKMan}). In fact, the numerical conditions in Theorem \ref{MainThm2} ensures that $R_e$ is a monodromy operator. The three cases appearing in Theorem \ref{MainThm2} are exactly the cases for which $e$ is the class of a prime exceptional divisor, that is $R_e\in W_\Exc$. 

	Particularly significant examples within  K$3^{[n]}$-type HK manifolds are moduli space of sheaves on K3 surfaces. Given a K3 surface $S$, one considers the lattice $\widetilde{H}(S,\bbZ)\simeq H^0(S,\bbZ)\oplus H^2(S,\bbZ)\oplus H^4(S,\bbZ)$ together with an intersection form defined by $(r,c,s)\cdot(r',c',s')=c\cdot c'-rs'-sr'$. Given a primitive algebraic Mukai vector $v\in \widetilde{H}(S,\bbZ)$ and $H\in\Pic(S)$ an ample line bundle, $M_H(S,v)$ denotes the Mumford-Gieseker moduli space of stable pure sheaves on $S$ with Mukai vector $v$ (we drop $S$ from the notation when the context is clear). For $H$ generic this is a HK manifold of K3$^{[n]}$-type when $v^2 \geq -2$, thanks to the work of many authors, notably Mukai, G\"ottsche, Huybrechts, O'Grady and ultimately Yoshioka \cite[Thm.\ 8.1]{Yos01}.

	\textbf{Notation:} Throughout this text, the pairing of a lattice $\Lambda$ is denoted by $(-,-)$ defined over the integers. We often conveniently denote $(w,w)$ by simply $w^2$. The saturation of a sublattice $\Lambda_1\subset \Lambda$ is denoted by $\overline{\Lambda_1}$. 
	Given an element $w\in \Lambda$, the \textsl{divisibility}, denoted $\dv(w)$, is defined to be the integer $n$ such that the image of the map $(\_, w)\colon \Lambda \to \bbZ$ is given by the ideal $n\bbZ$.
 
 K3 surfaces and HK manifolds considered here are projective. 
	For a HK manifold $X$ of K3$^{[n]}$-type,
	the lattice $(H^2(X, \bbZ), q_X)$ is isomorphic to $ \LambdaKn\coloneqq U^{\oplus 3} \oplus E_8(-1)^{\oplus 2}\oplus \langle -2(n-1)\rangle,$ where $U$ is the hyperbolic plane and $E_8$ is the unique even, unimodular, positive definite lattice of rank 8.
	The \textsl{Mukai Lattice}, denoted by $\MukaiLat$, is the abstract lattice $U^{\oplus 4}\oplus E_8(-1)^{\oplus 2}$ which is isomorphic to $\widetilde{H}(S,\bbZ)$ for any K3 surface $S$.
	
	Given a Hodge structure of weight 2 on a lattice $L$ (e.g.\ $\widetilde{H}(S,\matZ)$ or $H^2(X,\bbZ)$), we define the \textsl{algebraic part} of $L$ as $L^{1,1}_{\bbZ}\coloneqq  L^{1,1}_{\bbC}\cap L$. For instance, $\widetilde{H}(S, \bbC)$ comes with a polarized weight 2 Hodge structure, where the algebraic part 
 is given by the \textsl{extended N\'eron--Severi lattice}
	\[\widetilde{H}^{1,1}(S, \bbZ)\coloneqq H^0(S,\bbZ)\oplus H^{1,1}(S, \bbZ) \oplus H^4(S, \bbZ)\]
 and the polarization is induced by the intersection form described above.
	
	Finally, we define $\NS(X)$ as the Néron--Severi lattice and $\Tr(X)$ as the transcendental lattice of $X$, where $\Tr(X) \simeq \NS(X)^{\perp}$.

  \subsection*{Acknowledgements}
 Our gratitude goes to the anonymous referees who meticulously reviewed our paper and shared detailed and valuable input that significantly enhanced the quality of the exposition.
This article is indebted to the work of our predecessors in this area, especially to the results of Arend Bayer, Emanuele Macr\`i, Eyal Markman, and Giovanni Mongardi. We are especially grateful to Giovanni Mongardi for his invaluable input and ideas in initiating this project. We are thankful for numerous helpful conversations with Nick Addington, Pietro Beri, Marcello Bernardara, Luca Giovenzana, Klaus Hulek, Daniel Huybrechts, Giacomo Mezzedimi and \'Angel R\'ios.  We thank them for their inputs. 

The third named author was visiting the institutes of mathematics in Bonn, Hannover, and Toulouse, while working on this paper, and would like to thank these institutes for their excellent hospitality and stimulating atmosphere.
\subsection*{Funding}
The first named author was supported by the Hausdorff Center of Mathematics, Bonn under Germany’s Excellence Strategy [(DFG) - EXC-2047/1 - 390685813]; the second named author was supported by ERC Synergy Grant HyperK, agreement [ID 854361]; the third named author gratefully acknowledges support through a Riemann Fellowship at the Riemann Center for Geometry and Physics, Hannover; and the Laboratoire Ypatia des Sciences Math\'ematiques LYSM CNRS-INdAM International Research Laboratory.


 	\section{Cohomological action of symplectic birational maps}\label{SectionCohActionSymplecticBirMaps}

	The goal of this section is to prove Theorem \ref{MainThmSympBirImpliesModuli}.
The proof, that we give at the end of the section, is a combination of several results. First, we need a characterization of the birational models of $X$ via the lattice $H^2(X,\bbZ)$. In \cite[\S 9]{Mar11}, Markman describes $\widetilde{\Lambda}$ as an overlattice of $H^2(X,\hz)$ and an extension of the Hodge structure on $H^2(X,\hz)$ by setting $v\in \MukaiLat^{1,1}_{\bbZ}$. Furthermore, he shows  the following properties:

\begin{citedthm}[{(see \cite[Thm.\ 3]{Add16})}]\label{ThmEmbedH2toK3lattice} Let $X$ be a HK of K$3^{[n]}$--type, $n \geq 2$.
		\begin{enumerate}
			\item The orthogonal $H^2(X,\hz)^\perp \subset \widetilde{\Lambda}$ is generated by a primitive vector $v$ of square $2n-2$.
			\item If $X$ is a moduli space of sheaves on a K3 surface $S$ with Mukai vector $v \in \widetilde{H}(S,\hz)$, then the extension $H^2(X,\hz)\subset \widetilde{\Lambda}$ is naturally identified with $v^\perp \subset \widetilde{H}(S,\Z)$ (see also (\ref{MukaiIsoH2}) in \S \ref{SectionRemindersStabCond}).
			\item A HK manifold $Y$ is birational to $X$ if and only if there is a Hodge isometry $\widetilde{\Lambda} \function \widetilde{\Lambda}$ sending $H^2(X,\hz)$ isomorphically to $H^2(Y,\hz)$.
		\end{enumerate}
	\end{citedthm}

The action of a map $g\in \Bir(X)$ extends to a Hodge isometry on $\MukaiLat$. 
The following result is well-known. For the sake of completeness, we include it here.

	\begin{lemma} \label{Prop:extension of g on the Mukai lattice} Let $g\in \Bir(X)$. If $g^*_{|_{A_X}}= \varepsilon \Id$ ($\varepsilon=1$ or $-1$), there exists a unique extension of $g^*$ on $\MukaiLat$ satisfying $g^*v=\varepsilon v$.
	\end{lemma}
	
	\begin{proof}

		 Define an extension $g^*$ on $H^2(X,\matZ)\oplus \langle v \rangle$ by setting $v\mapsto \varepsilon v$. Since $H^2(X,\matZ)_\matQ\oplus \langle v \rangle_\matQ=\MukaiLat_\matQ$, it suffices to prove that the extension $g^*_\matQ \in \ho(\MukaiLat_\matQ)$ preserves $\MukaiLat$. Any element $x\in \MukaiLat =\MukaiLat^\vee \subset H^2(X, \bbZ)^\vee\oplus  \langle v\rangle^\vee$ decomposes as $x=x_1+x_2$, with $x_1 \in H^2(X, \bbZ)^\vee$ and $x_2\in \langle v\rangle^\vee$. 
		 Note that $g^*_\matQ(x_1)=\varepsilon x_1+y$ with $y\in H^2(X,\matZ)$, therefore $g_{\bbQ}^*(x)=\varepsilon (x_1+ x_2)+y\in \widetilde{\Lambda}$.
	\end{proof}

	Let $\widetilde{\Lambda}^{1,1}_{\bbZ}$ be the algebraic part of $\widetilde{\Lambda}$ (see \S \ref{secpreliminaries}). The next lemma demonstrates that the property of being birational to a moduli space of (twisted) sheaves on a K3 may be understood lattice-theoretically from the Hodge structure of $\widetilde{\Lambda}$. 
	
	The non-twisted case was proved by Addington in \cite[Prop.\ 4]{Add16} and later generalized for twisted K3 surfaces by Huybrechts in \cite[Lem.\ 2.6]{Huy17}. 

	\begin{theorem}\label{lemmaUGiveModuliOnK3} Let $X$ be a HK manifold of K$3^{[n]}$-type, $n\geq 2$. Then the following are equivalent:
		
		\begin{enumerate}
			\item $\widetilde{\Lambda}^{1,1}_{\bbZ}$ contains a copy of  $U$ (resp.\ $U(k)$).
			\item $X$ is birational to a moduli space of stable sheaves (resp.\ twisted sheaves) on a K3 surface $S$. 
		\end{enumerate}

	\end{theorem}

		Now let $g\in \Bir(X)$ be a non-trivial symplectic finite order map. Denote by $H^2(X,\hz)^{g^*}$ (resp.\ $\MukaiLat^{g^*}$) the invariant sublattice of $H^2(X,\hz)$ (resp.\ $\MukaiLat$). Similarly let $S_{g^*}(X)=(H^2(X,\hz)^{g^*})^\perp \subset H^2(X,\hz)$ (resp.\ $S_{g^*}(\MukaiLat)=(\MukaiLat^{g^*})^\perp \subset \MukaiLat$)  be the co-invariant sublattice of $H^2(X,\hz)$ (resp.\ $\MukaiLat$ ). The following is well-known.
		
		\begin{citedlem} \label{LemmaDescriptionCoinv}
		    $S_{g^*}(X)$ is a negative definite sublattice of $ \NS(X)$.
		\end{citedlem}
		\begin{proof}
		   Pick $x\in S_{g^*}(X)$. Let $\sigma\in H^{2,0}(X)$ denote a generator. Since $H^2(X,\matZ)^{g^*}$ is non-degenerate, $x'=\sum_{i=0}^{\ord(g)-1} (g^i)^*x$ is $g^*$-invariant and co-invariant, therefore $x'=0$. Since $(g^i(x),\sigma)=(x,\sigma)$ for all $i$, we obtain that $0=(x',\sigma)=\ord(g)(x,\sigma)$, therefore $(x,\sigma)=0$ and hence $S_{g^*}(X)\subset \NS(X)$. Recall that $\sign(\NS(X)) = (1, \star)$.
		
		  Let $\alpha \in H^{1,1}(X, \bbZ)$ be a K\"ahler class of $X$. The class $\alpha'=\sum_{i=0}^{\ord(g)-1} (g^i)^*\alpha$ is in $H^2(X,\hz)^{g^*}\cap \NS(X)$. 
		 As $(\alpha')^2>0$, we obtain $\sign(S_{g^*}(X))=(0,23-\rk H^2(X,\hz)^{g^*}).$
		\end{proof}
  
		More explicitly, the complex vector space $H^2(X,\bbC)^{g^*}$ admits a 3-dimensional positive subspace spanned by $\alpha', \sigma$ and $\overline{\sigma}$.
        As an upshot, we obtain that when $g^*|_{A_X}=-\Id$, the lattice $S_{g^*}(\widetilde{\Lambda})$ has signature $(1,23-\rk H^2(X,\hz)^{g^*})$, 
        i.e. it is a \textit{hyperbolic} lattice. On the other hand, when $g^*|_{A_X} = \Id$, we obtain $S_{g^*}(\widetilde{\Lambda}) \simeq S_{g^*}(X)$ and hence is again negative definite.

		\begin{remark}\label{RmkInvolution2ElemDiscGrp}
		   For a finite order group $G\subset \ho(L)$ of isometries of a lattice $L$, the quotient group $L/(S_G(L)\oplus L^G)$ is $|G|$-torsion. Indeed, for any element $x\in L$, we can write $|G|\cdot x= \sum_{g\in G}g(x) + \sum_{g\in G}(x-g(x))$. If $L$ is unimodular, we know moreover that $L/(S_G(L)\oplus L^G) \simeq A_{S_G(L)} \simeq A_{L^G}$. As a consequence, when the birational map $g\in \Bir(X)$ has order $2$, we obtain that
		   $$ A_{S_{g^*}(\MukaiLat)}\simeq (\matZ/2\matZ)^{\oplus \ell} $$
		   for some positive integer $\ell$. A lattice with such a discriminant group is $2$-elementary, and $2$-elementary lattices were classified by Nikulin in \cite[Thm.\ 4.3.2]{Nik81}.
		\end{remark}

	\begin{lemma}\label{LemmaIsoClassGiveU}
	An even lattice $L$ contains a non-trivial isotropic class $e\in L$ with $\dive(e)=m$ if and only if $U(m^2)\subset L$.
	\end{lemma} 

	\begin{proof}
	Pick an element $f'$ in $ L$ such that $(e,f')=m$ and $f'^2=2k$ for some $k\in\matZ$. Set $f=-ke+mf'$. It is straightforward to check that $\langle e,f \rangle=U(m^2) \subset L$. Conversely, if $U(m^2)=\langle e,f\rangle\subset L$, then $e$ is isotropic and one of its multiples has divisibility $m$.
	\end{proof}

	\begin{remark}\label{RmkLatticeIndefRk5containsIsoClass}
	   A consequence of a result by Hasse and Minkowski is that any indefinite lattice $L$ of rank at least $5$ must contain an isotropic class (see \cite[Cor.\ IV.3.2]{SerreACourseInArithmetic}) and hence by the lemma mentioned above $L$ contains a copy of $U(m^2)$ for some integer $m$. 
	\end{remark}

Let $N_{24}$ be the \textsl{Leech lattice}, i.e. the unique unimodular negative definite lattice of rank $24$ without $(-2)$-classes. The final ingredient we need is the following result that allows us to see
certain finite order maps in $\Bir(X)$ as elements of the so-called Conway group $Co_0\coloneqq  \ho(N_{24})$. The advantage of doing so is that one can use the classification of $S_G(N_{24})$ for any finite group action $G\subset Co_0$ as in \cite{HohnMason290FixedPt}. The original idea of using $N_{24}$ in this way is due to Gaberdiel, Hohenegger and Volpato \cite[B2]{GHRSymmK3SigmaModels}. We use the following more convenient form, which appears in \cite[Prop.\ 2.4]{GGOVSympRigidOG10}, see \cite[\S 2.2]{HuybrechtsDerivedCatK3ConwayGrp}.

\begin{lemma}\label{LemmaSgAsSubgroupOfConway}
	Let $L$ be a lattice and let $G\subset \ho(L)$ be a subgroup of isometries that acts trivially on the discriminant group $A_L$. Assume that $S_G(L)$ is negative definite and does not contain $(-2)$-classes. If there exists a primitive embedding $S_G(L) \subset \Lambda_{1,25}$ into the unique unimodular lattice of signature $(1,25)$, then $G$ is isomorphic to a subgroup of the Conway group $Co_0$ such that
	$$S_G(L)=S_G(N_{24}).$$
\end{lemma}

In order to use this, we need to prove the next property.

\begin{lemma}\label{LemmaNoRootsInSgX}
	Let $g\in \Bir(X)$ be a non-trivial symplectic birational self-map of finite order. Then there are no $(-2)$-classes in $S_{g^*}(X)$.  
\end{lemma}

\begin{proof}
Let $\delta\in \NS(X)$ be a $(-2)$-class in $S_{g^*}(X)$. By \cite[Thm.\ 9.17]{Mar11}, $\pm\delta$ or $\pm 2\delta$ is the class of an effective divisor $D$. Since $g$ is an isomorphism in codimension one,  $D'=\sum_{i=0}^{\ord(g)-1}g^i(D)$ is also an effective divisor. Its cohomology class $\sum_{i=0}^{\ord(g)-1}\pm(g^*)^i\delta$ or $\sum_{i=0}^{\ord(g)-1}\pm 2(g^*)^i\delta$ is both invariant and co-invariant, therefore $D'$ is homologous to $0$. Since linear and numerical equivalences coincide on HK manifolds, we obtain that $D'$ is linearly trivial, which is impossible.
\end{proof}

\begin{remark}
   In fact, the result is a consequence of a more general phenomenon. Namely, for a HK manifold $X$, the group $\MonB(X)$ preserves the (interior of the) movable cone, which consists of the classes $\beta$ for which  $(\beta,[E])>0$ for every stably prime exceptional class $[E]$ \cite[Prop. 1.8]{Mar11}. For $\alpha\in\Amp(X)$, the invariant class $\alpha'\coloneqq \sum_{i=1}^{\ord(g)}(g^i)^*\alpha$ lies in the interior of $\Mov(M)$. Given a $(-2)$-class $\delta$, either $\pm \delta$ of $\pm 2\delta$ is stably prime exceptional by \cite[Thm.\ 9.17]{Mar11}, so $(\delta,\alpha')\neq 0$. In particular, $\delta$ cannot lie in $S_{g^*}(X)$.
  
\end{remark}

We can now proceed to the proof of  Theorem \ref{MainThmSympBirImpliesModuli}.

\begin{proof}[Proof of Theorem \ref{MainThmSympBirImpliesModuli}]
    The crux of the proof lies in showing that $X$ is birational to a moduli of (twisted) sheaves on a K3 surface. The isomorphism follows from \cite[Section 7]{BMProjBirBridgelandModuli} (see also \S \ref{SectionReflectionAnyWall}).
    
	To do it, by Theorem \ref{lemmaUGiveModuliOnK3} and Lemma \ref{LemmaIsoClassGiveU}, it is enough to prove that $\MukaiLat^{1,1}_{\bbZ}$ admits an isotropic class. If $g$ is of order 2, and acts non-trivially on $A_X$, according to Remark \ref{RmkInvolution2ElemDiscGrp}, the co-invariant lattice $S_{g^*}(\MukaiLat)$ is a hyperbolic $2$-elementary lattice, and so the existence of such class follows by \cite[Thm.\ 4.3.3]{Nik81}. Otherwise, we show that $\widetilde{\Lambda}_{\bbZ}^{1,1}$ contains an indefinite lattice $L$ of rank at least five and use Remark \ref{RmkLatticeIndefRk5containsIsoClass}. 
    
    If $g^*|_{A_X} = -\Id$ and $g$ has order bigger than $2$, let $L = S_{g^*}(\widetilde{\Lambda})$. If $g^*|_{A_X} = \Id$, we let $L$ be the lattice generated by $S_{g^*}(X)$ and the vector $v$ as in Theorem \ref{ThmEmbedH2toK3lattice}. Then by Lemma \ref{Prop:extension of g on the Mukai lattice} $L$ is an indefinite sublattice of $\MukaiLat$ and we just need to show that it has rank at least 5.

	Set $f=g^2$ if $g^*|_{A_X}=-\Id$, or $f=g$ if $g^*|_{A_X}=\Id$, and $G=\langle f^* \rangle$ the finite group generated by $f^*$. Note that $S_{f^*}(X)$ is negative-definite by Lemma \ref{LemmaDescriptionCoinv}. Now, we will prove that $S_{f^*}(X)$ primitively embeds into the unique unimodular lattice $\Lambda_{1,25}$ of signature $(1,25)$. Since $S_{f^*}(X)$ embeds in the Mukai Lattice $\MukaiLat$, from Remark \ref{RmkInvolution2ElemDiscGrp} we have 
	\begin{eqnarray*}
	\ell(A_{S_{f^*}(X)}) = \ell(A_{S_{f^*}(\MukaiLat)}) &=&  \ell(A_{\MukaiLat^{f^*}})\\
	&\leq&  
	24-\rk S_{f^*}(\MukaiLat).
	\end{eqnarray*}
	We get $\ell(A_{S_{f^*}(X)})<\rk \Lambda_{1,25}-\rk S_{f^*}(\MukaiLat)$.
 This shows that $S_{f^*}(X) \subset \Lambda_{1,25}$ primitively by \cite[Cor.\ 1.12.3]{Nik79}.
 
  Finally, note that $f^*$ acts as $\Id$ on the discriminant group $A_X$ and by Lemma \ref{LemmaNoRootsInSgX}, the lattice $S_{f^*}(X)$ does not contain any ($-2$)-class. Therefore, we can apply Lemma \ref{LemmaSgAsSubgroupOfConway} to lift the action of $f^*$ to the Leech lattice $N_{24}$ and identify $S_{f^*}(X) = S_{f^*}(N_{24})$.
	Using this identification, along with the classification of fixed lattices for such groups in the Leech lattice $N_{24}$ as presented in \cite{HohnMason290FixedPt}, we can conclude that $S_{f^*}(X)$ has rank at least $24-16=8$. In particular, $S_{f^*}(X)\simeq S_G(N_{24})$ has order at least $24-16=8$. Since $S_{f^*}(X) \subset S_{g^*}(X) \subset L$, the bound of the rank applies to $L$.
\end{proof}

\begin{corollary}\label{CoroCoinvRkAtLeast8}
    Let $g\in \Bir(X)$ be a non-trivial symplectic birational self-map of finite order such that $\ord(g)>2$ or $g^*|_{A_X}=\Id$. Then $\NS(X)$ has rank at least 8.
\end{corollary}

\begin{proof}
Since $g$ acts symplectically on $X$, we have $S_{g^*}(X) \subset \NS(X)$ by Lemma \ref{LemmaDescriptionCoinv}. As shown in the proof of Theorem \ref{MainThmSympBirImpliesModuli}, for $g$ as in the statement we have $\rk S_{g^*}(X) \geq 8$. Consequently, the Picard number of $X$ is also at least 8.
\end{proof}

	\subsection{Markman's example $M_H(r,0,-s)$}\label{SectionMarkmanExamples} In this section we recall an explicit geometric example given by Markman \cite[sections 2.2, 11.1]{Mar13} of a symplectic birational involution on moduli spaces of sheaves on K3 surfaces with a non-trivial action on the discriminant group. 
	In the final part of the paper, where we study certain reflections coming from the so-called ``vertical wall'' (see Definition \ref{def:verticalwall}) of the stability manifold of the K3 surface, we revisit this example as it serves as one such reflection (see \S \ref{SectionMarkExStabCond}).

	Let $S$ be a projective K3 surface with $\Pic(S)=\matZ H$. Suppose that $r,s$ are two integers satisfying $s\geq r \geq 1$ and $\gcd(r,s)=1$. Set  $M\coloneqq  M_H(r,0,-s)$ the moduli space of stable sheaves on $S$ with Mukai vector $v=(r,0,-s)$. The Mukai isometry (\ref{MukaiIsoH2}) identifies $H^2(M,\matZ)$ with $H^2(S,\matZ)\oplus \matZ\cdot e$, where $e=(r,0,s)$.
	
	\begin{itemize}
		\item $\mathbf{r=1}$: In this case, $M_H(1,0,-s)=S^{[1+s]}$ and $e=[E]/2$ where $E \subset S^{[1+s]}$ is the diagonal, i.e. $E$ is the exceptional divisor of the Hilbert--Chow morphism $\varepsilon: S^{[1+s]} \function S^{(1+s)}$. 
		
		\item $\mathbf{r=2}$: In this case, $e=[E]$ where $E \subset M_H(2,0,-s)$ is the locus of stable sheaves which are not locally free. By \cite[Lem.\ 10.16]{Mar13}, $E$ is the exceptional divisor of Li's morphism \cite{Li93} from $M_H(2,0,-s)$ onto the Uhlenbeck--Yau compactification of the moduli space of $H$-slope stable vector bundles.
		
		\item $\mathbf{r\geq 3}$: In this case,  $e$ is not a $\hq$-effective class. Let $Z \subset M_H(r,0,-s)$ be the locus of stable sheaves which are not locally free or not $H$-slope stable. Denote by $U=M\setminus Z$ the locus of locally free $H$-slope stable sheaves and $\iota: U \function U$ the map that sends $\F$ to its dual sheaf $\F^\vee$. By \cite[Lem.\ 9.5]{Mar13}, $Z$ is a closed subset of codimension $\geq 2$ in $M$ and so $\iota: M \function M$ is a birational involution. By \cite[Prop.\ 11.1]{Mar13}, the induced map $\iota^*$ in cohomology corresponds to the reflection map $R_e$. Since $e\in\NS(X)$, the involution $\iota$ is in particular symplectic. 
	\end{itemize}

	\section{Picard rank $2$ case}\label{SectionPicRank2}
	
    Assume $X$ is a HK manifold of $K3^{[n]}$-type with Picard rank $2$. We know by Oguiso's result (c.f.\ \cite[Thm.\ 4.9]{DebarreHKMan}) that, while $\Aut(X)$ and $\Bir(X)$ can be infinite, whenever they are finite they are of the form $(\matZ/2\matZ)^r$ with $r\leq 2$. In this section, we prove Theorem \ref{PropPicRk2AllInvolutions} that states that any finite order birational self-map of $X$ is in fact of order $2$.

    \begin{proof}[{Proof of {Theorem \ref{PropPicRk2AllInvolutions}}}]
	Since $X$ has Picard rank $2$, its real N\'eron-Severi group $\NS_\mathbb{R}(X)$ is a $2$-plane. By a result of Markman \cite[Thm.\ 6.25]{Mar11}, there exist a polyhedral rational convex cone $\Delta\subset \Mov(X)$ such that
	$$ \overline{\Mov(X)} = \bigcup_{s\in\Bir(X)}s^*(\Delta).$$
    Let $\Id\neq g\in\Bir(X)$ be an element of 
    finite order. We can define
	$$ \Delta_g^0 \coloneqq \bigcup_{n\in\matZ} (g^{n})^*\Delta$$
	which is a finite union of convex polyhedral rational cones in $\overline{\Mov(X)}$. Define $\Delta_g$ as the convex hull of $\Delta_g^0$. Denote by $l_1,l_2$ the extremal rays of $\Delta_g$ (which lie already in $\Delta_g^0$). Then $g^*(\Delta_g)$ lies in the cone generated by $g^*(l_1)$ and $g^*(l_2)$, which lies in $\Delta_g$. By bijectivity of $g$, we get $g^*(\Delta_g)=\Delta_g$.
	
	Since the extremal rays of $\Delta_g$ are rational, we can choose integral generators $x_1$ of $l_1$ and $x_2$ of $l_2$. Then $g^*$ must either preserve or flip $l_1$ and $l_2$. In particular, $(g^2)^*(x_i)=\alpha_ix_i$ for some $\alpha_i>0$, $i=1,2$. Since $g^*$ is an integral isometry, we get $\alpha_1,\alpha_2\in\mathbb{Z}$ and $\alpha_1\alpha_2=1$. We get $\alpha_1=\alpha_2=1$ and hence $(g^2)^*=\Id|_{\NS(X)}$.

Since $\rk \Tr(X)$ is odd, the only integral Hodge isometries of $\Tr(X)$ are $\pm\Id$ (see \cite[\S 3.3]{Huy16}). Therefore $(g^2)^*|_{\Tr(X)}=(g^*|_{\Tr(X)})^2=\Id$, so in particular $g^*$ is an involution on $H^2(X,\matZ)$. Since the map $\Bir(X) \to \ho(H^2(X,\bbZ))$ is injective, we obtain that $g$ is an involution.

The second assertion follows from Corollary \ref{CoroCoinvRkAtLeast8}. Indeed, if $g$ is a symplectic automorphism, then $g^*|_{A_X} = \Id$ \cite[Thm.\ 26]{Mon16} and hence $\rk \NS(X) \geq 8$, which contradicts the hypothesis of the theorem. 
	\end{proof}
	
	\begin{remark}\label{RmkInvolutionAreReflections}
	 Another way to see the second part of Theorem \ref{PropPicRk2AllInvolutions} is as follows: any involution $\iota$ (different from $ -\Id$) on a rank $2$ lattice $L$ is a reflection. More precisely, pick any $x\in L$ with $x\neq \pm \iota(x)$ and consider  $y\coloneqq x+\iota(x)$ and $z\coloneqq x-\iota(x)$. We have $\iota(y) = y$ and $\iota(z)=-z$, so $\iota_\matQ$ and the reflection $R_z$ coincide on the basis $\{y,z\}$ of $L_\matQ$, hence coincide. Therefore, in the case of Picard rank $2$, a symplectic involution has to act as a reflection on $H^2(X,\bbZ)$. In Corollary \ref{CoroSympReflectionActsNonTrivDiscGrp}, we deduce that this involution cannot be biregular.
	\end{remark}

	\begin{remark}
	     In \cite[Prop.\ 4.19]{DebarreHKMan}, Debarre describes a family of polarized fourfolds whose birational groups are isomorphic to the infinite dihedral group. In particular, order $2$ elements exist despite the birational groups being infinite.
	\end{remark}

    	\section{Involutions for arbitrary Picard ranks}
    	
The goal of this section is to show the existence of symplectic finite order birational self-maps on certain HK manifold $X$ of K$3^{[n]}$-type as described in Theorem \ref{MainThm2}. According to the Torelli theorem, a way to proceed is to find monodromy operators of $H^2(X,\matZ)$. In \S \ref{SectionPicRank2} we proved that when $\Pic(X)$ has rank $2$, any such finite order operator is a reflection (see Remark \ref{RmkInvolutionAreReflections}). Therefore, one candidate in higher Picard rank are reflections $R_e\in \ho(H^2(X,\matZ))$ in the hyperplane orthogonal to some class $e\in\NS(X)$. We will use the following result of Markman.

\begin{citedprop}[{\cite[Prop.\ 1.5]{Mar13}}]\label{PropMarkmanConditionReIsMonodromyOp}
    Let $e\in H^2(X,\matZ)$ be a primitive class 
    with $e^2<0$. Then the reflection $R_e$ belongs to $\Mon(X)$ if and only if $e$ satisfies one of the following:
    \begin{enumerate}
        \item $e^2=-2$, or
        \item $e^2=2-2n$ and $n-1$ divides $(e,-)\in H^2(X,\matZ)^\vee$.
    \end{enumerate}
    Moreover, $R_e$ acts on $A_X$ as $\Id$ in (1)
    and $-\Id$ in (2).
\end{citedprop}

\begin{corollary}\label{CoroSympReflectionActsNonTrivDiscGrp}
   If a symplectic birational involution $g\in\Bir(X)$ acts as a reflection $g^*=R_e$ on $H^2(X,\matZ)$, then necessarily $e^2=2-2n$. In particular, $g^*|_{A_X}=-\Id$.
\end{corollary}
\begin{proof}
  Note that $e\in S_{R_e}(X)$ since for any $w\in H^2(X,\bbZ)^{R_e}$ we have $(e,w) = (-e, w)$. Moreover $e^2<0$ by Lemma \ref{LemmaDescriptionCoinv}. Then $e^2\neq -2$ by Lemma \ref{LemmaNoRootsInSgX} . 
\end{proof}

Furthermore, we will focus on the moduli space of sheaves $M$ on K3 surfaces and use certain walls in the stability manifold of the K3 surface to find these reflection hyperplanes in $\Mov(M)$.

	\subsection{Reminders on stability conditions}\label{SectionRemindersStabCond}
	
Let $S$ be a K3 surface and $\alpha\in\Br(S)$ be a Brauer class (see \cite[Section 10.2.2, 16.4]{Huy16} for details), and fix $B\in H^2(S,\matQ)$ a $B$-field lift of $\alpha$. Let $D^b(S,\alpha)$ denote the derived category of $\alpha$-twisted sheaves on $S$. We will not give the general definition of \textsl{stability conditions} (we redirect the interested reader to \cite{MacriSchmidtLecturesBridgelandStab} for the general theory, and we refer to \cite{BridgelandStabCondK3Surfaces} for the case of K3 surfaces), but only recall how to construct the ones we will use in the present paper and some of their properties.

 Pick a pair of $\matR$-divisors $\omega,\beta\in\NS(S)_\matR$, with $\omega\in\Amp(S)$, the ample cone inside $\NS(S)_{\bbR}$. For $E\in\D^b(S,\alpha)$, define
 \begin{eqnarray}\label{DefGeometricStabCond}
	Z_{\omega,\beta}(E)&=&(\exp(i\omega+\beta+B),v(E))\in \bbC  
 \end{eqnarray}
In particular, writing $v(E)=(r,\theta,s)$, we can write the imaginary part $\Im Z_{\omega,\beta}(E)=\omega\cdot(\theta-r(\beta+B))$.
Define a pair of additive subcategories (called a \textsl{torsion pair}) $\calT^{\omega,\beta},\calF^{\omega,\beta}\subset \D^b(S,\alpha)$ as follows: the non-trivial objects of $\calT^{\omega, \beta}$ (resp. $\calF^{\omega,\beta}$) are the twisted sheaves (resp. torsion free twisted sheaves) $F\in \Coh(S,\alpha)$ such that every non-trivial torsion free quotient $F\twoheadrightarrow G$ (resp. non-zero subsheaf $E\hookrightarrow F$) satisfy $\Im Z_{\omega,\beta}(G)>0$ (resp. $\Im Z_{\omega,\beta}(E)\leq 0$). Then define the abelian subcategory of $\D^b(S,\alpha)$:
\begin{eqnarray}
	\calA_{\omega,\beta}\coloneqq \{F\in\D^b(S,\alpha) \ | \ \calH^{-1}(F)\in \calF^{\omega,\beta}, \calH^0(F)\in \calT^{\omega,\beta}, \calH^i(F)=0 \text{ for }i\neq -1,0 \}
\end{eqnarray}

The following result is due to Bridgeland \cite{BridgelandStabCondK3Surfaces} in the untwisted case, and it was later generalized to the twisted case by Huybrechts, Macr\`i, Stellari \cite[Prop. 3.6]{HuybMacriStellariStabCondGenericK3Cat}.
\begin{theorem}
	The pair $\sigma_{\omega,\beta}=(\calA_{\omega,\beta},Z_{\omega,\beta})$ is a stability condition on $\D^b(S,\alpha)$ provided that $Z_{\omega,\beta}(E)\notin \matR_{\leq 0}$  for all spherical twisted sheaf $E\in\Coh(S,\alpha)$. This condition is satisfied whenever $\omega^2>2$. Moreover, the stability condition $\sigma_{\omega,\beta}$ depends continuously on $(\omega,\beta)\in\Amp(S)\times\NS(S)_\matR$.
\end{theorem}

We denote by $\Stab^{\dagger}(S)$ the component of the space of stability condition $\Stab(S)$ containing all stability conditions of the form $\sigma_{\omega,\beta}$.
We drop the $\omega,\beta$ when the context is clear. Define the \textsl{generalized rank} as $\rk_Z = \Im Z : \calA\to \matR_{\geq 0}$ and \textsl{generalized degree} as $\deg_Z = -\Re Z : \calA \to \matR$. These maps satisfy the usual properties of rank and degree of sheaves, such as $\rk(E)=0\Rightarrow \deg(E)>0$ and additivity with respect to exact sequences. Therefore one can define the \textsl{generalized slope} as
$$\mu_Z(E)\coloneqq  \deg_Z(E)/\rk_Z(E),$$
where $\mu_Z(E)=\infty$ when $\rk_Z(E)=0$.
Objects $E\in\D^b(S,\alpha)$ such that $E[k]\in\calA$ for some $k\in\matZ$ and $E[k]$ is slope-(semi)stable with respect to $\mu_Z$ are called $Z$-(semi)stable (or $\sigma$-(semi)stable).

\subsubsection{Walls and moduli spaces}

Fix a Mukai vector $v\in \widetilde{H}(S,\alpha,\matZ)$. 

\begin{citedthm}[{\cite[Prop.\ 3.3]{BayerMacriSpaceStabCondLocalProjPlane}}]
	There exists a locally finite set of \textsl{walls} (real codimension one submanifolds with boundary) in $\Stab(S)$, depending on $v$, such that the sets of $\sigma$-stable and $\sigma$-semistable objects of class $v$ do not change when $\sigma$ varies in a chamber (i.e.\ in a component of the complement of the walls).
\end{citedthm}

A stability condition is called \textsl{generic} with respect to $v$ if it does not lie on a wall. When $v$ is primitive, then $\sigma$ lies on a wall if and only if there exists a strictly $\sigma$-semistable object of class $v$.

In \cite{BMProjBirBridgelandModuli}, the authors prove that for a generic $\sigma\in\Stab^{\dagger}(S)$ there exists a coarse moduli space $M_\sigma(v)$ of $\sigma$-semistable objects with Mukai vector $v$. It is a normal projective irreducible variety with $\matQ$-factorial singularities. When $v$ is primitive, then $M_\sigma(v)=M_\sigma^{st}(v)$, i.e.\ it consists only of $\sigma$-stable objects and hence is a smooth projective HK manifold of K$3^{[n]}$-type.

\begin{citedthm}[{\cite[Thm.\ 2.15]{BMMMPwallcrossing}}]
	Let $v=mv_0\in\widetilde{H}(S,\alpha,\matZ)$ be a vector with $v_0$ primitive and $m>0$, and let $\sigma\in\Stab^{\dagger}(S,\alpha)$ be a generic stability condition with respect to $v$.
	\begin{enumerate}
		\item The coarse moduli space $M_\sigma(v)$ is non-empty if and only if $v_0^2\geq -2$,
		\item Either $\dim M_\sigma(v)=v^2+2$ and $M_\sigma^{st}\neq \emptyset$, or $m>1$ and $v_0^2\leq 0$.
	\end{enumerate}
\end{citedthm}

\begin{remark}
	Given a polarization $H\in \Pic(S)$ which is generic with respect to $v$, there is always a \textsl{Gieseker chamber} in $\Stab^\dagger(S)$: the moduli space $M_\sigma(v)$ for any $\sigma$ in this chamber is isomorphic to the moduli space of $H$-Gieseker stable sheaves, see \cite[Prop. 14.2]{BridgelandStabCondK3Surfaces}.
\end{remark}

The Hodge structure on the second cohomology of $M_\sigma(v)$ is closely related to the one of $S$ by the so-called \textsl{Mukai Hodge isometry}
\begin{eqnarray}\label{MukaiIsoH2}
\theta:H^2(M_\sigma(v),\matZ) \to 
\begin{cases}
v^\perp  \subset \widetilde{H}(S,\alpha, \matZ) & \text{ if } v^2\neq 0 \\
v^\perp/v  \subset \widetilde{H}(S,\alpha, \matZ) & \text{ if } v^2= 0.
\end{cases}
\end{eqnarray}
The latter has been shown by Yoshioka \cite[Sections 7 \& 8]{Yos01} for moduli of Gieseker stable sheaves, and generalized by Bayer and Macr\`{i} \cite[Thm.\ 7.10]{BMProjBirBridgelandModuli} to Bridgeland stability conditions. This way, one can identify the groups $\NS(M_{\sigma}(v))$ with $\NS(M_{\tau}(v))$, for any two $v$-generic stability conditions $\sigma$ and $\tau$.

\subsubsection{Wall-crossing}

Bayer and Macr\`i proved in \cite{BMMMPwallcrossing} that the MMP for a moduli space of Gieseker stable sheaves $M_H(v)$ can be performed by wall-crossing in the space of stability conditions; namely, they proved the following.

\begin{citedthm}[{\cite[Thm.\ 1.1, 1.2]{BMMMPwallcrossing}}]\label{ThmBMAllBirModelAreModuliStab}
	Let $v$ be a primitive Mukai vector.
	\begin{enumerate}
		\item Given $\sigma,\tau\in\Stab^{\dagger}(S)$ generic, then the two moduli spaces $M_\sigma(v)$ and $M_\tau(v)$ of Bridgeland-stable objects are birational to each other.
		\item Fix a (generic) base point $\sigma\in\Stab^{\dagger}(S)$. Every smooth $K$-trivial birational model of $M_{\sigma}(v)$ appears as a moduli space $M_\tau(v)$ for some $\tau\in\Stab^{\dagger}(S)$.
	\end{enumerate}
\end{citedthm}

\noindent To be more precise, the authors construct a map 
\begin{eqnarray}\label{MapStabToNS}
l:\Stab^{\dagger}(S) \to \NS(M_{\sigma}(v)),
\end{eqnarray} 
such that for any chamber $\calC\subset \Stab^{\dagger}(S)$ and $\tau\in\calC$ we have $l(\calC)=\Amp(M_\tau(v))$. Given a wall $\calW$ for $v$ and $\sigma_0$ a generic stability condition on the wall (namely, which does not belong to any other wall), let $\sigma_+,\sigma_-$ be two $v$-generic stability conditions nearby $\calW$ in opposite chambers. 

\begin{citedthm}[{\cite[Thm.\ 1.4]{BMProjBirBridgelandModuli}}]\label{ThmBMContractionWall}
	There exist birational contractions
	$$\pi^{\pm} : M_{\sigma_\pm}(v) \to \overline{M}_\pm$$
	where $\overline{M}_\pm$ are normal irreducible projective varieties. The curves contracted by $\pi^\pm$ are precisely the curves of objects that are $S$-equivalent to each other with respect to $\sigma_0$.
\end{citedthm}

The type of birational transformation is described as follows.

\begin{definition}\label{DefWallsTypes}
	We call a wall $\calW$:
	\begin{enumerate}
		\item a \textit{fake wall} if there are no curves in $M_{\sigma_\pm}(v)$ of objects that are $S$-equivalent to each other with respect to $\sigma_0$,
		\item a \textit{flopping wall} if we can identify $\overline{M}_+=\overline{M}_-$ and the induced map $M_{\sigma_+} \dashrightarrow M_{\sigma_-}$ induces a flopping contraction,
		\item a \textit{divisorial wall} if the morphisms $\pi^{\pm}:M_{\sigma_\pm}(v) \to \overline{M}_\pm$ are both divisorial contractions.
	\end{enumerate}
\end{definition}

The type of wall can be studied purely lattice-theoretically. Given a wall $\calW$, one can associate the hyperbolic rank $2$ primitive sublattice $\calH_\calW\subset \widetilde{H}(S,\alpha,\matZ)$ given by
$$\calH_\calW = \{w\in \widetilde{H}(S,\alpha,\matZ) \ | \ \Im \dfrac{Z(w)}{Z(v)}=0 \textit{ for all } \sigma=(\calA,Z)\in\calW \}.$$
Conversely, given a primitive rank $2$ hyperbolic sublattice $\sH$ containing $v$, one can define a \textit{potential wall} $\calW$ associated to $\calH$ as a connected component of the real codimension one submanifold of stability conditions $\sigma=(\calA,Z)$ which satisfy that $Z(\calH)$ is contained in a line. 

\begin{citedthm}[{\cite[Thm.\ 5.7]{BMMMPwallcrossing}}]\label{ThmLatticeTypeWall}
	Let $\calH\subset \widetilde{H}(S,\alpha,\matZ)$ be a primitive hyperbolic rank $2$ lattice containing $v$ and $\calW\subset \Stab^{\dagger}(S)$ a potential wall associated to it.
	\begin{enumerate}
		\item[(1)] The set $\calW$ is a divisorial wall if there exists a class $s\in\calH$ such that $(s,v)=0$ and $s^2=-2$, or an isotropic class $w\in\calH$ such that $(w,v)=1$ or $2$.
		\item[(2)] The set $\calW$ is a flopping wall if the conditions in $(1)$ are not satisfied and either $v$ can be written as $v=a_1+a_2$ with $a_i\in\calH$, $a_i^2\geq 0$ and $(a_i,v)>0$, $i=1,2$, or  there exists a class $s\in\calH$ with $s^2=-2$ and $0<(s,v)\leq v^2/2$.
		\item[(3)] In all other cases, $\calW$ is either a fake wall or it is not a wall.
	\end{enumerate}
\end{citedthm}

\begin{remark}\label{RmkWallTypesWallPositions}
    Following \cite[Lem. 10.1]{BMMMPwallcrossing}, the type of a wall is directly related to the position of its image by (\ref{MapStabToNS}) in $\overline{\Mov}(M_\sigma(v))$: the image $l(\calW)$ lies on the boundary of $\Mov(M_\sigma(v))$ if and only if it is divisorial. On the other hand, $\calW$ is either a flopping or fake wall if its image lies in the interior of $\Mov(M_\sigma(v))$ (it is fake if and only if its image lies in the interior of a chamber corresponding to a birational model of $M_\sigma(v)$).
\end{remark}
	
	\subsection{Reflection in the vertical wall}\label{SectionVerticalReflections}

        Set $M\coloneqq  M_\sigma(S,v)$ for some K3 surface $S$ with arbitrary Picard rank and $\sigma\in\Stab^\dagger(S)$ be a $v$-generic condition. Denote
        $$v=(r,\theta,s)$$
        with $r,s\in \matZ$ and $\theta\in\NS(S)$. For convenience, we denote $\theta=cD$ with $c\in\matZ$ and $D$ primitive. We will freely use the Mukai isometry (\ref{MukaiIsoH2}) to identify $H^2(M,\matZ)$ with $v^\perp\subset \widetilde{H}(S,\matZ)$.

        If $\Phi:\D^b(S)\xrightarrow{\sim} \D^b(S)$ is an autoequivalence, the induced isometry $\varphi\in \ho(\widetilde{H}(S,\matZ))$ on the extended Mukai lattice gives an isomorphism
        $$M_\sigma(v) \xrightarrow{\sim} M_{\Phi(\sigma)}(\varphi(v)).$$
        An additional useful observation is that by Theorem \ref{ThmBMAllBirModelAreModuliStab} (and assuming $\Phi(\sigma)\in\Stab^\dagger(S)$), for any $\tau\in\Stab^\dagger(S)$ we have a birational map
        $$M\simeq M_{\Phi(\sigma)}(\varphi(v)) \overset{\sim}{\dashrightarrow} M_{\tau}(\varphi(v)). $$    
        In particular, it holds for $\tau$ in the Gieseker chamber, whenever it exists.

        We will use freely the following three types of autoequivalences:
        \begin{itemize}
            \item The shift $[-1]$, which acts as $-\Id$ on $\widetilde{H}(S,\matZ)$.
            \item The tensor product $-\otimes L$ for $L\in\Pic(S)$, which acts as cup product with $\exp(L)=(1,L,L^2/2)$ on $\widetilde{H}(S,\matZ)$.
            \item The spherical twist $T_{\calO_X}$, which acts as the reflection $R_{(1,0,1)}$ in the hyperplane orthogonal to the spherical class $(1,0,1)\in \widetilde{H}(S,\matZ)$.
        \end{itemize}
    From \cite[Lem. 7.2, Prop. 7.6]{HartmannCuspsKaehlerModuliStabCondK3}, these autoequivalences preserve the distinguished component $\Stab^\dagger(S)\subset \Stab(S)$. In particular, applying some of these equivalences we assume $r>0$ from now on. 
        \newline

        Consider the geometric stability conditions of the form $\sigma=\sigma_{\omega,\beta}$ with $\omega\in\Amp(S), \beta\in \NS(S)_\matR$
        defined by (\ref{DefGeometricStabCond}) with $\alpha=B=0$.
        \begin{definition}\label{def:verticalwall}
            We call the potential wall $\calW\subset \Amp(S)\times\NS(S)_\matR$ (seen as a subspace of $\Stabd(S)$) defined by the equation
            $\Im Z_{\omega,\beta}(v)=0$
            the \textsl{vertical wall} of $S$ associated to $v$.
        \end{definition}
       
        In other words, $\calW$ is spanned by the set of $(\omega,\beta)$ satisfying the equation $\omega\cdot\theta = r\omega\cdot \beta$. Consider the corresponding hyperbolic rank $2$ lattice
        \begin{equation}\label{eq:hyperbolicW}
            \calH_{\calW}= \{w=(\lambda,(\lambda/r)\theta,\mu), \ \lambda,\mu \in \matZ, \ r\mid c\lambda \},
        \end{equation}
        i.e. the saturation of the sublattice spanned by $v$ and $(0,0,1)$, and note that $Z_{\omega,\beta}(w)\in \matR$ for all $(\omega,\beta)\in\calW$ and $w\in\calH_\calW$.

        Using \cite[Lem.\ $9.2$]{BMProjBirBridgelandModuli}, it is easy to see that the map $l$ as in (\ref{MapStabToNS}) sends any stability condition $\sigma_{\omega,\beta}\in\calW$ (up to a constant) to the vector $a_{\omega,\beta}=(0,\omega,\beta\cdot \omega)=(0,\omega,\omega\cdot\theta/r)$. These vectors form a cone $C_\calW$ of codimension one (i.e.\ a wall) in the positive cone $\mathcal{C}_M \subset \NS(M)_\matR\simeq v^\perp_\matR\subset \widetilde{H}^{1,1}(S)_\matR$.
               In fact, one of the chambers adjacent to $C_\calW$ is the Gieseker chamber, which is a consequence of the \textsl{large volume limit}. More precisely, as proven in \cite[Prop. 14.1]{BridgelandStabCondK3Surfaces}, objects that become $\sigma_{n\omega,\beta}$-semistable for $\sigma_{n\omega,\beta}$ close enough to (one side of) the vertical wall and for all $n$ big enough are exactly (shifts of) Gieseker-semistable sheaves. Moreover, computations in \cite[Thm. 3.11]{MaciociaComputingWalls} show that there is no wall-crossing for $n\gg0$, so that we can choose an uniform $n$ for all objects.

        We consider the reflection in this wall. To do so, consider the vector $e\coloneqq (r,\theta,(\theta^2/r)-s)$. It is orthogonal both to $v$ and $a_{\omega,\beta}$ for all $(\omega,\beta)\in \calW$. In particular, $e$ is a generator of the line orthogonal to the wall $C_\calW\subset \Mov(M)$.
        Finally, note that $e^2=-v^2=2-2n$. We now use Markman's criterion in Proposition \ref{PropMarkmanConditionReIsMonodromyOp} to check when the reflection $R_e$ is a monodromy operator.
        
       \begin{lemma}\label{LemmaReInMon2conditions}
        The reflection $R_e$ belongs to $\Mon(M)$ with $R_e\mid_{A_M}=-\Id$ if and only if
         \[
   r\mid 2c  \ \text{ and } \ \gcd(r,s)=1 \text{ or } 2.\label{conditionrHighRank} \tag{$\ast$} 
   \]
        \end{lemma}
        
        \begin{proof}

Note that if $R_e \in \Mon(M)$ and $R_e|_{A_M}=-\Id$, then $e$ is an integral primitive class. Indeed, write $e=ke_0$ for some $k \in \matQ$ and $e_0$ integral primitive. Then, $R_e=R_{e_0}$ and Proposition \ref{PropMarkmanConditionReIsMonodromyOp} gives $e_0^2=2-2n=e^2=k^2e_0^2$, hence $k=\pm 1$. Thus in fact, $R_e \in \Mon(M)$ and $R_e|_{A_M}=-\Id$ if and only if $e$ is an integral primitive class and $n-1|\dive_{H^2(M,\matZ)}(e)$.
Integrality of $e$ is equivalent to $r\mid c^2D^2$ and primitivity is equivalent to $\gcd(r,c,c^2D^2/r-s)=1$. Therefore, it is enough to show that (\ref{conditionrHighRank}) is equivalent to the numerical conditions $r\mid c^2D^2$, $\gcd(r,c,c^2D^2/r-s)=1$ and $n-1\mid \dive_{H^2(M,\matZ)}(e)$.

First, since $D^2$ is even, we have (\ref{conditionrHighRank}) implies $r\mid cD^2$ (that is, $e$ integral), and since $\gcd(r,c,s) =\gcd(r, c, cm - s)$ for $m=\frac{cD^2}{r}$ and $v$ is primitive, we also have that $e$ is primitive. 

We claim that once $e$ is integral and primitive, Condition (\ref{conditionrHighRank}) is equivalent to the condition $n-1\mid \dive_{H^2(M,\matZ)}(e)$. 
To see the claim, recall that integrality of $e$ implies $r\mid c^2D^2$ and for what follows, we set $c^2D^2 = kr$ for some $k\in\matZ$. 
Pick some $w = (a, \delta, f) \in v^\perp \subset \widetilde{H}(S,\matZ)$ for $a, f \in \bbZ$ and $\delta\in H^2(S,\bbZ)$.
Recall that there is an inclusion $H^2(S,\matZ) \hookrightarrow \NS(S)^\vee \oplus T(S)^\vee$. Therefore, we can assume that $w$ has the shape
$$w=(a,b\dfrac{L}{\dv(L)}+T,f).$$
for some integers $a,b$ and classes $L\in\NS(S)^\vee$ and $T\in T(S)^\vee$. Here and in the following $\dv(L)\coloneqq\dive_{\NS(S)}(L)$.
The condition $w\in v^\perp $ becomes
\begin{equation}
    bc\dfrac{L\cdot D}{\dv(L)}=as+rf
    \label{eq:winvperp}
\end{equation} and a direct computation gives
$(w,e)=\dfrac{2a}{r}(1-n)$. Therefore,
the condition $(n-1)|(w,e)$ is equivalent to $r|2a$. We will first assume this holds for all $w$ and prove that it implies (\ref{conditionrHighRank}).

Choose a basis $\{L_1,\dots,L_m\}$ of $\NS(S)$, where $m=\rk \NS(S)$. Write $D=\sum a_jL_j$. For each $k=1,\dots,m$, consider 
$$K_k\coloneqq \bigcap_{j\neq k}(L_j)^\perp_\matQ \subset \NS(S)_\matQ$$
which is a non-trivial vector subspace. Let $A_k\in K_k$ non-trivial. Taking some multiple of $A_k$, we can assume $A_k\in \NS(S)$.
Then for each $E=\sum b_jL_j\in \NS(S)$, we have $E\cdot A_k= b_kL_k\cdot A_k$. In particular, $\dv(A_k)=A_k\cdot L_k$. Moreover, $D\cdot A_k=a_k L_k\cdot A_k$. We obtain
$$\frac{D\cdot A_k}{\dv(A_k)}=\frac{a_k L_k\cdot A_k}{L_k\cdot A_k}=a_k.$$
To show $r\mid 2c$, we pick $w_k=(ca_k-r, s\dfrac{A_k}{\dv(A_k)}+T,s)\in v^{\perp}$. By the assumption $r\mid 2a$, for each $k$, we get $r\mid 2(ca_k-r)$. Since $D$ is primitive, $\gcd(a_1,\dots,a_m)=1$ and hence we get $r\mid 2c$.

Now, write $g=\gcd(r,s)$ and pick $w=(r/g,2D,(2c/r)D^2-s/g)$. We get $w\in v^\perp$ and the condition $r\mid 2a$ becomes $2r = grm$ for some $m\in \bbZ$. Therefore, $g\mid 2$ and we have (\ref{conditionrHighRank}).

Conversely, assume (\ref{conditionrHighRank}) holds and take $w=(a,b\frac{L}{\dive(L)}+T,f)$ as before. From (\ref{eq:winvperp}) we have
\begin{eqnarray}\label{Eq2asFinal}
2as=r\left( \dfrac{2c}{r}\dfrac{bD\cdot L}{\dv(L)}-2f\right).
\end{eqnarray}
Now if $g=1$, then $r\mid 2a$. If $g=2$, primitivity of $v$ implies $c$ odd, and $r\mid 2c$ implies $r/2$ odd. Dividing both sides of equation (\ref{Eq2asFinal}) by $2$ yields $r/2 \mid 2a(s/2)$, so $r/2 \mid 2a$ so $r/2\mid a$ and therefore $r\mid 2a$.

\end{proof}

From now on, we assume that $v$ satisfies (\ref{conditionrHighRank}). Since $R_e\in \MonH(M)$, by \cite[Thm. 1.6]{Mar11}, it decomposes uniquely (uniqueness is a consequence of the injectivity of $\Bir(M) \to \ho(H^2(M,\matZ))$) as 
 $R_e= \psi\circ g^*$,
with $g\in \Bir(M)$ a symplectic birational self-map and $\psi\in W_\text{Exc}$.

     \begin{proposition}\label{prop:walldivisorialorflopping}
         The potential wall $\calW$ is either a divisorial or a flopping wall. 
     \end{proposition}
     
     \begin{proof}
     
     All we have to prove is that $\calW$ is an actual non-fake wall. If $\calW$ is a fake wall, or not an actual wall, its image $C_\calW$ would lie inside of the interior of a chamber corresponding to the image of the nef cone of a birational model $f:M\dashrightarrow M'$. In particular, $C_\calW$ lies in the interior of $\Mov(M)$ and therefore $\psi=\Id$: indeed, one can find $D\in \Mov(M)$ such that $(g^{-1})^*(D)$ lies close to $C_\calW$. Then we have $R_e\circ (g^{-1})^*(D)=\psi(D)\in \Mov(M)$. This is only possible when $\psi=\Id$ since $\Mov(M)$ is a fundamental domain for the action of $W_{\Exc}$ and $W_\Exc$ acts faithfully (\cite[Lem. 6.22]{Mar11}). We get that $R_e=g^*$ preserves $f^*\Nef(M')$. The Torelli theorem then implies that the map $f\circ g\circ f^{-1}$ extends to an isomorphism on $M'$ (see \cite[Prop. $3.15$]{DebarreHKMan}), which renders $g^*|_{A_M}=\Id$  \cite[Thm.\ 26]{Mon16} contradicting our assumption.

     \end{proof}

Using Theorem \ref{ThmLatticeTypeWall} we can also deduce under condition (\ref{conditionrHighRank}) when $R_e$ does not come from a birational involution. In other words, we give a necessary and sufficient numerical criterion for when $\sW$ is a divisorial wall.

    \begin{proposition}\label{PropConditionVertDivisorialWall}
    Under the condition (\ref{conditionrHighRank}), the vertical wall as in Definition \ref{def:verticalwall} is a divisorial wall if and only if one of the following cases occurs.
    \begin{itemize}
        \item $r=1, r=2$ or $v = (r, krD, \dfrac{D^2k^2r }{2}-m)$ for some $k\in\matZ$ and  $m = 1$ or 2.
        \item  $r>2$, $r\nmid c$, and one of the two possibilities occurs.
        \begin{enumerate}
    	\item $D^2\equiv 0 \pmod{4}$, $v=(2a,maD,\dfrac{D^2m^2a}{4}-1)$ 
    	for some integers $a\geq 2$, $m$ odd. 
    	\item $D^2\equiv 2 \pmod{4}$, $v=(2a,m'aD,\dfrac{D^2m'^2a-2}{4})$  
    	for some integers $a\geq 3$ odd, $m'$ odd.
        \end{enumerate}

    \end{itemize}
    \end{proposition}

    \begin{proof}[{Proof of Proposition \ref{PropConditionVertDivisorialWall}}]

Note that any Mukai vector $w=(\lambda,(\lambda c/r)D,\mu)\in\calH_\calW$ satisfies the following formulas, which we use in the proof:
    \begin{eqnarray}
    & & w^2 = \lambda^2\dfrac{D^2c^2}{r^2}-2\lambda\mu \label{w2} \\
    & & (w,v)=\lambda\dfrac{D^2c^2}{r}-r\mu-s\lambda \label{wv}.
    \end{eqnarray}

    By Theorem \ref{ThmLatticeTypeWall} it is enough to give a complete classification of Mukai vectors $v$, for which there exists an isotropic $w\in \sH_{\sW}$ satisfying $(w, v) = 1$ or 2 or a class $w\in\sH_{\sW}$ with $w^2 = -2$ satisfying $(w,v) = 0$. To this end, we first prove that if such a $w$ exists, then it imposes restrictions on $v$, leaving us with the above list in the statement. In the process, we also obtain necessary values of $\lambda$ and $\mu$ for each instance of $v$, proving the existence of such $w$'s in the converse direction.   
    \newline
    
    \hspace{1cm}  \textbf{Step 1:} Assume $r=1$ or $2$.
    
    Then $w=(0,0,-1)$ satisfies $w^2=0$ and $(w,v)=1$ or $2$, so the vertical wall is divisorial.
    \newline
    
            \hspace{1cm} \textbf{Step 2:} Assume $c=kr$ for some $k\in\matZ$.

    Let $w=(\lambda,\lambda \dfrac{c}{r}D,\mu)$ as above. If $\lambda=0$, then $w^2=0$ and $(w,v)=\mu r = 1$ or $2$ if and only if $r=1,2$. Hence we assume $\lambda\neq 0$.

    Equations (\ref{w2}) and (\ref{wv}) become 
    \begin{eqnarray*}
    w^2&=&\lambda^2 D^2k^2-2\lambda\mu \\
    (w,v)&=&\lambda D^2 k^2 r - r\mu - s\lambda
    \end{eqnarray*}
As discussed above there are two ways, $\sW$ could be a divisorial wall for $v$: 
    
    $\bullet $ either such a $w$ satisfies  $w^2=0$ and $(w,v) = 1,2$. Then $\lambda D^2k^2=2\mu$. Hence
    \begin{eqnarray*}
    (w,v) &=& \lambda D^2k^2r-r\mu -s\lambda \\
    &=& \lambda D^2k^2r - r\lambda k^2 (D^2/2) -s\lambda \\
    &=& \lambda((D^2/2)k^2r -s).
    \end{eqnarray*}
   Now $(w,v)=1$ implies $\lambda=\pm 1$ and $(w,v)=2$ implies $\lambda=\pm 1,\pm 2$. It is direct to see that $\lambda<0$ leads to $v^2<0$, which is absurd (recall we assumed $v^2\geq 2$).     Hence, we obtain $s=(D^2/2)k^2r -m$ with $m=1$ or $2$. Hence $v$ is forced to be of the form $(r, krD, \dfrac{D^2k^2r }{2}-m)$ in this case. Conversely, for any such $v$, the isotropic vector $w = (1, kD, \frac{D^2k^2}{2})$ satisfies $(w,v)=m$.

    $\bullet$ Or, such a $w$ satisfies $w^2=-2$ and $(w,v) = 0$, i.e.\ $\lambda^2 D^2k^2=2\lambda\mu-2$, so we obtain again $\lambda=\pm 1$ or $\pm 2$. 
    In fact, $\lambda=\pm 2$ is impossible because we would obtain $4D^2k^2=\pm 2(2\mu -1)$. Hence $\lambda=\pm 1$ and
    $$(w,v)=0 \iff  \pm D^2k^2r=r\mu\pm s, $$
    so $r|s$ which implies that $r=1$ since $v=(r,krD,s)$ is primitive. So we are back to Step 1.
    \newline

    \hspace{1cm} \textbf{Step 3:} General case.

    By Condition (\ref{conditionrHighRank}), we can assume that $cD^2=kr$ and $c\lambda=mr$ for some $k,m$. Note that as in Step 2, when $\lambda=0$ the wall is divisorial only when $r=1$ or $2$, and $\lambda =\pm 1$ gives $r\mid c$, hence these cases are already treated in Step $1$ and $2$. Therefore we assume $|\lambda|\geq 2$.
    Equalities (\ref{w2}) and (\ref{wv}) give:
    \begin{eqnarray*}
     w^2&=&\lambda m k - 2 \lambda \mu \\
    (w,v)&=&rmk - r\mu -s\lambda
    \end{eqnarray*}

    First, assume $w^2=-2$. Then $\lambda=\pm2$. The wall is divisorial in this case, if and only if $(w,v) = 0$. Hence the second equality implies $r\mid 2s$, which gives $r=1$ or $2$ and we are back to Step $1$.
    
    Now assume $w^2=0$, that is $mk=2\mu$ and hence $(w,v) = r\mu - s\lambda$. 
    Since $2\mu r = mrk = \lambda ck$, we obtain
    \begin{eqnarray*}
    2(w,v) = \lambda(ck-2s).
    \end{eqnarray*}
    Assume $(w,v)=1$. Then $ck-2s=\pm 1$, in particular $c,k$ are odd. If $r$ is even then $ck=\pm \mu r$ is even, which leads to a contradiction. Since by Condition (\ref{conditionrHighRank}) $r\mid 2c$, if $r$ is odd, we obtain $r\mid c$, so we are done by Step $2$.
    
    The last case is $(w,v)=2$, which gives $\lambda=\pm 2$ or $\pm 4$. Note that $\lambda<0$ is impossible, since it would give $ck-2s<0$, but $v^2=r(ck-2s)\geq 0$. 
    
    First, let $\lambda=2$. We obtain $ck-2s= 2$. Since by Condition (\ref{conditionrHighRank}) $r\mid 2c$, we can assume $r=2a$ for some $a\geq 2$ (otherwise $r\mid c$ and we are back to Step $2$), therefore $c=ma$ with $m$ odd and $k=mD^2/2$. We obtain
    $$s=\dfrac{ck}{2} - 1 = \dfrac{D^2}{4}m^2 a  -1.$$
    Note that $\mu=mck/r=m^2aD^2/4a=m^2D^2/4$, in particular $D^2\equiv 0 \pmod{4}$. For any such choice of $a$ and $m$, the integral vector $w=(2,mD,m^2D^2/4)$ gives $(w,v)=2$. 
    
    The situation with $\lambda=4$ is similar: we get $ck-2s=1$, in particular $c,k$ are odd. As before we can assume $r=2a$ for some $a$, but moreover, if $a$ is even, we get $2\mu\mid ck$, which is impossible. Therefore $a\geq 3$ is odd, $m=2m'$ is even (with $m'$ odd), $c=m'a$, $k=m'D^2/2$ and
    $$ s=\dfrac{m'^{2}a(D^2/2) -1}{2},$$
    which in particular implies that $D^2\equiv 2 \pmod{4}$. For any such choice of $a$ and $m'$, the integral vector $w=(4,2m'D,\dfrac{D^2}{2}m'^{2})$ gives $(w,v)=2$.

    \end{proof}
    
    We conclude that whenever $v$ does not satisfy the condition of the previous proposition, $\calW$ is a flopping wall and $R_e\in \MonB(M)$. Proposition \ref{PropConditionVertDivisorialWall} together with Proposition \ref{prop:walldivisorialorflopping} gives a proof of Theorem \ref{MainThm2}.

    \begin{proof}[\bf Proof of Theorem \ref{MainThm2}]
     The reflection $R_e$ decomposes as $R_e=\psi\circ g^*$ with $g\in \Bir(M)$ is a symplectic birational self-map and $\psi\in W_\text{Exc}$. In view of Proposition \ref{prop:walldivisorialorflopping}, all we have to show is that $\psi=\Id$ exactly when $\calW$ is a flopping wall. We use Remark \ref{RmkWallTypesWallPositions}. If $\calW$ is divisorial, then it lies on the boundary of $\Mov(M)$. From \cite[Thm. 6.17 and Lem. 6.22]{Mar11}, $\calW=E^\perp$ for some prime exceptional divisor $E$, in particular $R_e=\psi$ lies in $W_\Exc$ and $g=\Id$. On the other hand, when $\calW$ is a flopping wall, it lies in the interior of $\Mov(M)$. The same argument as in Proposition \ref{prop:walldivisorialorflopping} shows that in this case $\psi=\Id$, i.e. $R_e=g^*$. The equality $g^*|_{A_M}=-\Id$ is a consequence of Proposition \ref{PropMarkmanConditionReIsMonodromyOp}.
    \end{proof}
    
    Given a Mukai vector $v$ as in Theorem \ref{MainThm2}, it would be interesting to know an explicit geometric description of the birational involution $g$. We discuss in \S \ref{SectionMarkExStabCond} one such explicit case.

    \subsection{Reflection along any wall}\label{SectionReflectionAnyWall}

    In this subsection, $\calW$ is any wall in $\Stab(M)$ that induces a monodromy operator with non-trivial action on the discriminant group $A_M$. The goal is to show that reflection along the image of $\sW$ under $l$ may be seen as reflection along the image of the vertical wall on an isomorphic model of $M$, as considered in the previous subsection. We need the following key construction.

    \subsubsection{Construction} \label{constructionFM} Consider a wall $\calW$ as above, and let $v=(r,\theta,s)$ as in \S \ref{SectionVerticalReflections}. Assume that the associated hyperbolic lattice $\calH_\calW$ contains a non-zero isotropic class $w$,  which implies that $\calH_{\calW}=\overline{\langle v,w \rangle}$. For such an isotropic class, we let $S' \coloneqq  M_H(S, w)$. We let $\Phi: \D^b(S) \to \D^b(S', \alpha)$, for some Brauer class $\alpha$ on $S'$, be the induced Fourier--Mukai transform with $\Phi(w)=(0,0,1)$, and set $v'=\Phi(v)$.
    Define $\sW'\subset \Stab(S')$ to be the induced wall for $\Phi(w)$, i.e.\ defined by the hyperbolic plane $\sH_{\sW'}=\overline{\langle \Phi(w), v'\rangle}$. Let $v' = (r', \theta', s')$. Pick $\sigma_0\in\calW$ and set $\tau_0\coloneqq \Phi(\sigma_0)$. Note that $\tau_0$ has the shape $\tau_0=\sigma_{\omega',\beta'}$ for some $\omega',\beta'\in\NS(S')_\matR$ (see \cite[Prop. 7.5]{HartmannCuspsKaehlerModuliStabCondK3} for $\alpha$ trivial, or \cite[Prop. 5.2]{HuybrechtsDerivedAbelianEquivK3Surf}). Since $\Phi(w) = (0,0,1)$ we get $\Im Z_{\omega', \beta'}(\Phi(w)) = 0$. Therefore $\Im Z_{\omega',\beta'}(v')=0$, i.e. $\sW'$ is the vertical wall of $S'$. 
    
    Moreover, $-r'=(v',(0,0,1))=(v,w)$. Composing the equivalence $\Phi$ with the shift $[1]$ we get for any $\sigma\in\Stabd(S)$ an isomorphism $$M=M_\sigma(S,v)\xrightarrow{\sim} M'=M_{\tau}(S',-v',\alpha)$$
    for $\tau\coloneqq \Phi(\sigma)[1]$ and the image of $\sW$ is the vertical wall of $S'$ with respect to $-v'$. Note that if $-r'=1$ then necessarily $\alpha$ is trivial:  this is because for any $\alpha\in\Br(S')$ and torsion-free $\alpha$-twisted sheaf $F$, the endomorphism algebra $\calE nd(F)$ (seen as an Azumaya algebra) represents the class $\alpha$ (see for instance \cite[Thm. 1.3.5]{CaldararuPhD}). When $F$ has rank $1$, $\calE nd(F)$ is the trivial bundle, and so $\alpha$ is trivial.
    
    An easy upshot of this construction is that it simplifies the Mukai vectors in  Proposition \ref{PropConditionVertDivisorialWall} significantly when we allow ourselves to choose a possibly different K3 surface $S$.
    
    \begin{corollary}\label{verticalwallonetwo}
       Up to replacing $S$ by a twisted Fourier--Mukai partner $(S',\alpha)$, under the condition (\ref{conditionrHighRank}) the only cases for which the vertical wall $\calW$ is divisorial are the cases $r=1,2$.
    \end{corollary}

    \begin{proof}
        Consider the cases in Proposition \ref{PropConditionVertDivisorialWall}. First, when $c=kr$ and $s=(D^2/2)k^2r -m$ for $m=1$ or 2, replacing $v$ with its image $v'$ under the  composition of $(-)\cdot \exp(-kD)$ and $R_{(1,0,1)}$ gives $v'=(m,0,-r)$. It is straightforward to check that these isometries send the vertical wall associated to $v$ to the one associated to $v'$.
        For the case $v=(2a,maD,\frac{D^2}{4}m^2a-1)$, resp. $v=(2a,maD,\frac{D^2m^2a-2}{4})$, the vector $w=(2,mD,\frac{D^2}{4}m^2)$, resp. $(4,2mD,\frac{D^2}{2}m^2)$ is an isotropic class in $\calH_\calW$ and satisfies $(w,v)=2$. Therefore, the construction in \S \ref{constructionFM} above applies. The rank of the new Mukai vector is $-r'= (v,w) = 2$.
    \end{proof}

    Another implication of the construction above is the following.
    \begin{corollary}\label{CoroAllReflectionAreVertical}
        Assume that the reflection $R_e$ in the wall $l(\calW)$ orthogonal to the vector $e\in \NS(M)$ lies in $\MonH(M)$ and acts non-trivially on $A_M$. Then, up to replacing $S$ by a twisted Fourier--Mukai partner $(S',\alpha)$, $R_e$ is the reflection in the vertical wall. In particular, for a birational involution $g\in \Bir(M)$ with $g^*=R_e\in \ho(H^2(M,\matZ))$, up to a derived equivalence $g^*$ is the reflection in the vertical wall.
    \end{corollary}
 
 \begin{proof}
    As we have seen in the proof of Theorem \ref{MainThm2}, an element $R_e\in \MonH(M)$ lies either in $\MonB(M)$ or $W_\text{Exc}$. In the former case, $e^2<0$ by Lemma \ref{LemmaDescriptionCoinv}. In the latter case, $e^2<0$ by \cite[Thm. 6.18]{Mar11}. By Proposition \ref{PropMarkmanConditionReIsMonodromyOp}, we have $e^2=2-2n$, where $\dim M=v^2+2=2n$, in particular, $v+e\in \calH_\calW$ is an isotropic class. Therefore, the construction applies and we can assume that $\calW$ is the vertical wall.
  
 \end{proof}

    \begin{remark}\label{RmkContractionVerticalWallIsLiUhlenbeck}
       The contraction map $\pi:M\to \overline{M}$ of Theorem \ref{ThmBMContractionWall} induced by the vertical wall is the map to the Uhlenbeck compactification constructed by Li \cite{LiAlgGeomInterpretDonaldsonPoly}, see \cite[Section 8]{BMMMPwallcrossing}. When $r=1$, an isometry given by applying a combination of $\exp{(-kD)}$ for some $k\in \bbZ$ and the reflection $R_{(1,0,1)}$
       reduces $v$ to the case when $v=(1,0,1-n)$, i.e.\ $M$ is birational to a Hilbert scheme of points. In this case, the map $\pi$ is nothing but the Hilbert--Chow morphism, see \cite[Example $10.1$]{BMProjBirBridgelandModuli}.
    \end{remark}

	\section{Some Examples}

    \subsection{Markman's example redux}\label{SectionMarkExStabCond}

    In this section, we revisit in the light of stability conditions the reflection maps discussed in \S \ref{SectionMarkmanExamples}. We show that the reflection map $R_e$ associated to some particular $e$ can be obtained as a reflection on the vertical (flopping) wall in the halfplane of stability conditions.
	
	Let $S$ be a K3 surface with $\Pic(S)=\matZ\cdot H$, $H$ ample. Pick $v=(r,0,-s)$ with $s\geq r\geq 1$ and $\gcd(r,s)=1$. Set $M\coloneqq M_H(v)$. The image of the vertical wall in $\Stab^{\dagger}(S)$ in $H^2(M,\matZ)=v^\perp$ is generated by $e=(r,0,s)$. We see that $v$ satisfies the conditions (\ref{conditionrHighRank}), in particular, the reflection $R_e$ lies in $\Mon(M)$. We consider the same trichotomy as in \S \ref{SectionMarkmanExamples}.

	\begin{itemize}
	    \item $\mathbf{r=1}$: In this case, $M\simeq S^{[1+s]}$. The wall is easily seen to be divisorial (of type \textit{Hilbert-Chow}). Indeed, the vector $w=(1,0,0)$ is an isotropic vector satisfying $(w,v)=1$. In particular, by Remark \ref{RmkContractionVerticalWallIsLiUhlenbeck} the contraction morphism induced on the wall is the Hilbert-Chow morphism $S^{[1+s]}\to S^{(1+s)}$.
	    
	    \item $\mathbf{r=2}$: In this case, the vector $w=(0,0,-1)$ gives $(w,v)=2$. Hence by Theorem \ref{ThmLatticeTypeWall} the vertical wall is divisorial (of type Li--Gieseker--Uhlenbeck), and by Remark \ref{RmkContractionVerticalWallIsLiUhlenbeck} the contraction morphism induced on the wall is the morphism onto the Uhlenbeck--Yau compactification of the moduli space of $H$-slope stable vector bundles.
	    
	    \item $\mathbf{r>2}$: Since $s> r$, according to Proposition \ref{PropConditionVertDivisorialWall}, the vertical wall is not divisorial. However, the class $w=(1,0,1)$ satisfies $w^2 = -2$ and $0<(w,v)\leq (v,v)/2$, hence $\calW$ is a flopping wall. The birational map $M\dashrightarrow M, \calF \mapsto \calF^{\vee}$ (for $\calF$ that is $H$-slope stable and locally free sheaf) induces the reflection $R_e$ in cohomology.
	\end{itemize}

	\subsection{Hilbert schemes}\label{SectionExampleHilbScheme}

        Let $S$ be a K3 surface with $\Pic(S)=\matZ H$, $H^2=2d$, $d>1$, and let $M=S^{[n]}$. In this set-up, the authors in \cite{BeriCattaneoBiratTransfHilbSchemeK3Surf} have given a precise description of $\Bir(M)$ for all values of $n,d$. In some cases, $M$ admits a symplectic birational involution $\sigma$ acting non-trivially on $A_M$. In these cases, $\sigma^*$ acts as a reflection on $H^2(M,\matZ)$, and hence by Corollary \ref{CoroAllReflectionAreVertical}, up to replacing $S$ by a Fourier-Mukai partner, this reflection is in fact the reflection in the vertical wall. In what follows, we can construct in some cases an explicit isometry to relate birationally the involution $\sigma$ on $M$ with Markman's example described in \S \ref{SectionMarkmanExamples} and in \S \ref{SectionMarkExStabCond}.

        Pick any $r>3$, set $n=r^2d-r+1$, so that $M=S^{[n]}=M_H(v)$ with $v=(1,0,r(1-rd))$. Interestingly, in this case, the vertical wall is divisorial. Nonetheless, by \cite[Thm.\ 1.1]{BeriCattaneoBiratTransfHilbSchemeK3Surf} $M$ admits a symplectic birational involution $\sigma$ and the line $\ell\subset \NS(M)$ generated by $2rdH_M  -2d\delta$ is fixed by $\sigma$, where $H_M=(0,H,0)$ (a polarization of $M$) and  $\delta=(-1,0,r(1-rd))$ (half the class of the big diagonal) generate $\NS(M)$. Consider the isometry of $\widetilde{H}(S,\matZ)$ given by the composition
        $$ \varphi = \exp(-H)\circ -R_{(1,0,1)} \circ \exp(-rH).$$
        By direct computations, one can check that $\varphi(v)=(r,0,1-rd)$, so we obtain a birational map $\widetilde{\varphi}:M\dashrightarrow M'\coloneqq  M_H(r,0,1-rd)$ by Theorem \ref{ThmEmbedH2toK3lattice}.
        Note that the image of the fixed line $\ell$ in $\NS(M')$ is generated by 
        $(0,H,0)$, so the action of the composition 
        \begin{eqnarray}\label{EqnRelSnAndMarkEx}
         \psi\coloneqq \widetilde{\varphi}\circ \sigma \circ \widetilde{\varphi}^{-1}
         \end{eqnarray}
        induces an action on $\NS(M')$ just as in Markman's example described \S \ref{SectionMarkmanExamples} and \S \ref{SectionMarkExStabCond}. By injectivity of the map $\Bir(M') \to \ho(H^2(M',\matZ))$, the birational map $\psi$ is precisely given by $M'\dashrightarrow M', \ \calF \mapsto \calF^\vee$.
        
        Explicit geometric descriptions of the involutions classified in \cite{BeriCattaneoBiratTransfHilbSchemeK3Surf} are not easy to give (see \cite{BeriManivelABiratInvol}). It would be interesting to give a geometric interpretation of $\sigma\in \Bir(S^{[n]})$ from Markman's example using (\ref{EqnRelSnAndMarkEx}).
        
        \begin{remark}
           We should point out that this example is somewhat special: in view of \cite[Rmk. 9.21]{Mar11}, two moduli spaces $M_H(r_i,0,-s_i)$, $i=1,2$ with $r_1s_1=r_2s_2$ are not birational in general. 
        \end{remark}
        
        As a matter of course, the example leads to the following
    \begin{question}
        Is $S^{[n]}$ birational to $M_H(r,0,-s)$ with $s>r>2$ whenever it admits a symplectic birational involution?
    
    \end{question}

\end{document}